\documentclass{ip-journal}

\usepackage{graphicx}
\usepackage{amsmath,latexsym}
\usepackage{amsfonts,amssymb}
\usepackage{amscd, amsthm}
\usepackage{stmaryrd}
\usepackage{fancyhdr}
\usepackage{commath,enumitem,chngcntr}
\usepackage{indentfirst, hyperref}
\usepackage[titletoc,toc,title]{appendix}
\usepackage{pst-node}
\usepackage{tikz-cd} 
\usepackage{pgf,tikz}
\usepackage{mathrsfs}
\usetikzlibrary{arrows,calc,positioning}
\usepackage{float}

\newtheorem{theorem}{Theorem}[section]
\newtheorem{lemma}[theorem]{Lemma}

\theoremstyle{definition}
\newtheorem{definition}[theorem]{Definition}

\theoremstyle{remark}

\numberwithin{equation}{section}

\begin{document}

\title{Comparing Shapes of High Genus Surfaces}

\author{Yanwen Luo}
\address{Department of Mathematics, University of California Davis, Davis
, California 95616}
\email{ywluo@ucdavis.edu}
\thanks{The author was supported in part by NSF Grant DMS-1719582.}


\begin{abstract}
In this paper, we define a new metric structure on the shape space of a high genus surface. We introduce a rigorous definition of a shape of a surface and construct a metric based on two energies measuring the area distortion and the angle distortion of a quasiconformal homeomorphism. We show that the energy minimizer in a fixed homotopy class is achieved by a quasiconformal homeomorphism by the lower semicontinuity property of these two energies. 
\end{abstract}

\maketitle

\section{Introduction}
How to measure the difference between two shapes is a fundamental problem in computer graphics, computer vision, and medical imaging. In this paper, we investigate the shape comparison problem between two shapes of a surface of genus at least one by considering the following questions:
\begin{enumerate}
	\item What is the precise meaning of a ``shape" of a surface?
	\item How similar are two given shapes of a surface?
	\item How to construct the best global alignment of two shapes, namely an ``optimal" correspondence between two shapes?
\end{enumerate}
We always assume that the surfaces we discuss have genus at least one. We will define the shape space as the space of equivalence classes of Riemannian metrics on a fixed smooth surface, up to isometries isotopic to the identity on the surface.
We show that the shape space has a close connection with the Teichm\"uller
 space of a surface. Then we construct a metric on the shape space by introducing an energy for quasiconformal homeomorphisms of the surface, which measures the similarity of two shapes. We show that the infimum of this energy in a fixed homotopy class is achieved by a quasiconformal homeomorphism, which produces the ``optimal" correspondence between two shapes and realizes the distance between two shapes.

This problem has been studied extensively in the fields of surface registration, shape matching, shape morphing, and texture mapping. Effective algorithms have been developed if the topology of the surface is relatively simple, such as with the 2-dimensional disk or 2-dimensional sphere \cite{ gotsman2003fundamentals, gu2004genus, Hass2017}. However, there are few results about the computation of optimal maps between high-genus surfaces \cite{li2008globally, lui2014geometric, wong2014computation, zeng2012computing}. On the other hand, detecting the change of the shapes of high genus surfaces is crucial to understanding various applications. For example, the vestibular system in the inner ear is modelled by a genus-three surface, and the morphometry of the vestibular system has been an active research field in the analysis of Adolescent Idiopathic Scoliosis Disease\cite{wen2015landmark}. In the study of deformity of the vertebrae, the vertebrae bone is modeled by a genus-one surface\cite{lam2015landmark}.

Comparing shapes of high genus surfaces is much more challenging than the case of the 2-sphere.  Any two metrics on the 2-sphere are conformal to each other, but for high-genus surfaces, conformal maps are insufficient to measure the difference between two shapes. Algorithmically, the main difficulty is how to deal with the topology of the surfaces. One possible approach is to construct local injective maps from disk-like patches to some canonical domain and glue them to form a global map. This method requires a consistent way to cover the whole surface with patches. An alternative method is to cut the surface using a system of disjoint loops to a disk-like surface, but boundary conditions on the loops are not natural. 

The key to measuring the difference between two surfaces is finding a metric structure on the shape space of a surface. More precisely, for a metric $d$ defined on the shape space, given shapes $F_1$, $F_2$, and $F_3$, we require the following properties:
 \begin{enumerate}
	\item $d(F_1, F_2) \geq 0$;
	\item $d(F_1, F_2) = 0$ if and only if $F_1$ and $F_2$ represent the same element in the shape space;
	\item $d(F_1, F_2) = d(F_2, F_1)$;
	\item $d(F_1, F_2) + d(F_2, F_3) \geq d(F_1, F_3)$.
\end{enumerate}
These properties of metric structures imply that we can distinguish two different shapes if the two shapes are not isometric, independent of the order and stable under small perturbations or noise.

The main result of this paper gives a metric structure $d$ on the shape space of a high genus surface $\mathcal{S}(F)$, based on the energy $E(f)$ for a quasiconformal homeomorphism $f$ between two shapes. More precisely, we prove the following theorem in Section 4 of this paper. 
\begin{theorem}
	Let $F$ be a closed orientable connected surface of genus $g\geq 1$. The function $d$ induces a metric on the space of shapes $\mathcal{S}(F)$. Moreover, for any pair of shapes $(F, g_1)$ and $(F, g_2)$, there exists a quasiconformal homeomorphism $f: (F, g_1) \to (F, g_2)$ such that $E(f) = d((F, g_1), (F, g_2))$.
\end{theorem}

This paper is organized as follows. In Section 2, we summarize previous work related to the computation of special maps between surfaces and some necessary mathematical background. In Section 3, we define the space of shapes $\mathcal{S}(F)$ and establish its connection with the Teichm\"uller
 space. In Section 4, we introduce an energy $E(f)$ for a quasiconformal homeomorphism $f$ on a surface and prove that this energy provides a metric $d$ on the shape space $\mathcal{S}(F)$, and an ``optimal" correspondence between two shapes. 

\section{Prior Work and Preliminary}
In this section, we summarize related work about various definitions of shape spaces and computational methods to find maps between surfaces. We assume that we have a closed connected orientable surface, a genus-zero surface $\mathbb{S}^2$ or high-genus surface $F$ with genus $g\geq 1$. Here we focus on three types of well-known classes maps between surfaces: conformal maps, harmonic maps, and quasiconformal maps. A comprehensive survey about surface parametrization using these maps can be found in Floater and Hormann\cite{floater2005surface}. 

\subsection{Conformal maps and Harmonic maps}
Conformal maps are the most familiar maps among these three maps. In the smooth theory, the fundamental result is the \textit{Uniformization Theorem} (see e.g.\cite{imayoshi2012introduction}).
\begin{theorem}
Every Riemannian metric on a closed surface $F$ is conformally equivalent to a complete Riemannian metric with constant curvature +1, 0, or -1, the sign depending on the sign of its Euler characteristic $\chi(F)$. The metric is unique up to isometry isotopic to the identity if the Euler characteristic is negative.
\end{theorem}

The general theory of harmonic maps between two $n$-dimensional manifolds was developed by Eells and Sampson\cite{eells1964harmonic}. We restrict our attention to the case of surfaces. The \textit{Dirichlet energy} of a map between two surfaces $f: (F_1, m_1) \to (F_2, m_2)$ is defined by
$$E_D(f) = \int_{F_1} ||df||^2 dA$$
where $df$ is the differential of $f$, considered as a section to the bundle $T^*F_1\otimes TF_2$ with a metric induced from $m_1$ and $m_2$. It can be regarded as the measurement of total stretching of the map $f$. A map is \textit{harmonic} if it is a critical point of the Dirichlet energy among maps in its homotopy class.

 One of the earliest results about harmonic maps in the plane is the Rado-Kneser-Choquet theorem\cite{duren2004harmonic}.
\begin{theorem}
	Suppose $\phi:\mathbb{D}\to \mathbb{R}^2$ is a harmonic map sending the boundary $\partial \mathbb{D}$ homeomorphically into the boundary $\partial\Sigma$ of some convex region $\Sigma \subset \mathbb{R}^2$. Then $\phi$ is one to one.
\end{theorem}
When it comes to general surfaces, a fundamental question is the existence and uniqueness of harmonic maps in a given homotopy class of maps between two surfaces. Here we summarize the results proved by Jost\cite{jost1982existence}, Schoen and Yau\cite{schoen1978univalent}, Coron and Helein\cite{coron1989harmonic}, and Markovic and Mateljevic\cite{markovic1999new}.
\begin{theorem}
Given two Riemannian metrics on a surface $F$ and a diffeomorphism $f$, there exists a diffeomorphism which is a critical point of the Dirichlet energy in the homotopy class of $f$. If the genus of $F$ satisfies $g>1$, then this diffeomorphism is unique.
\end{theorem}

Tons of results have been developed to compute conformal maps and harmonic maps between surfaces \cite{wu2020computing, sun2015discrete}. Among these results, the theory of discrete conformal geometry including circle packings and vertex scaling provide solid mathematical foundations and effective algorithms to compute discrete conformal maps between polyehedral surfaces \cite{stephenson2005introduction, gu2018discrete, gu2018discrete2, gu2019convergence, luo2020discrete, wu2015rigidity, wu2014finiteness, wu2020convergence}. 
\subsection{Quasiconformal maps and the Teichm\"uller maps}
Quasiconformal maps provide a generalization of conformal maps between surfaces, arising naturally when we want to compare two conformal structures on a surface. Let $f: D\to \mathbb{C}$ be an orientation preserving diffeomorphism from a region $D$ in $\mathbb{C}$. We can consider the \textit{Beltrami coefficient}
$$\mu_f(z) = \frac{f_{\bar{z}}}{f_z}. $$
If $f$ is conformal, then $f_{\bar{z}} = 0$ so $\mu_f  = 0$. The Jacobian of $f$ is given by $J(f) = |f_z|^2 - |f_{\bar{z}}|^2$ which is positive by assumption. Hence $|\mu_f|$ varies from $0$ to $1$, measuring the deviation of $f$ from a conformal map. An alternative quantity $K$ varying from $1$ to $\infty$, called the \textit{dilatation}, is defined by
$$K_f(z) =  \frac{1 + |\mu_f|}{1 - |\mu_f|}.$$
Geometrically, at each point $z$ in $D$, $df$ maps circles in $T_pD$ to ellipses in $T_{f(p)}\mathbb{C} = \mathbb{C}$. The dilatation $K_f(z)$ is the ratio of the major axis to the minor axis of the ellipse. Then we call the map $f$  a \textit{$K$-quasiconformal map} if there exists a $K>0$ such that
$$\sup_{z\in D} K_f(z) = \sup_{z\in D} \frac{1 + |\mu_f(z)|}{1 - |\mu_f(z)|} \leq K. $$
The composition of a $K_1$-quasiconformal map with a $K_2$-quasiconformal is a $K_1K_2$-quasiconformal map. Quasiconformal maps can be generalized further to non-differentiable maps using several mutually equivalent geometric and measure-theoretic definitions \cite{imayoshi2012introduction}.

\begin{figure}
\begin{center}
\begin{tikzpicture}
\draw [->] (0,-1.2) -- (0,1.2);
\draw [->] (-1.2,0) -- (1.2,0);
\draw [red] (0,0) circle [radius=0.8];
\draw [blue, rotate=45] (0,0) ellipse (36pt and 16pt);
\draw [->] (0,0) -- (0.9, 0.9);
\draw [->] (0,0) -- (-0.4, 0.4);
\node  at (1.4, 1.2) {$1 + |\mu|$};
\node  at (-0.9, 0.9) {$1 - |\mu|$};
\node  at (1, -1) {$K = \frac{1 + |\mu|}{1 - |\mu|}$};

\end{tikzpicture}
\caption{Quasiconformal maps}
\end{center}
\end{figure}
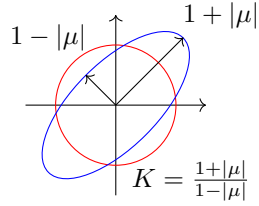

Every quasiconformal map $f: D \to \mathbb{C}$ gives rise to a Beltrami coefficient $\mu_f(z)$ defined on $D$. A remarkable theorem proved by Ahlfors and Bers (see, e.g.\cite{farb2011primer}) states the converse is also true.
\begin{theorem}
 If $\mu \in L^{\infty}(\mathbb{C})$ and $||\mu||_{\infty} < 1$, there exists a unique quasiconformal  homeomorphism $f: \hat{\mathbb{C}} \to \hat{\mathbb{C}}$ fixing $0$, $1$, and $\infty$, satisfying  $\mu = \mu_f$ almost everywhere.
\end{theorem}

Since the composition of quasiconformal maps with conformal maps is again quasiconformal with the same maximal dilatation, we can define quasiconformal maps $f: F_1 \to F_2 $ between Riemann surfaces using local charts. Then the Beltrami coefficient is a $(-1, 1)$-form $\mu d\bar{z}/dz$ instead of a function,  but $|\mu|$ is well-defined on the surface. We can define the corresponding dilatation of a map $f$ 
$$K_f = \sup_{p \in F} \frac{1 + |\mu_f(p)|}{1 - |\mu_f(p)|}. $$
This quantity measures the difference between two conformal structures, or equivalently, two hyperbolic structures for higher genus surfaces.  We have the following extremal problem in a given homotopy class: find a map $f_0$ achieving this infimum of the dilatation in a homotopy class satisfying
$$\ K_{f_0} = \inf \{K_f| f \text{ in a given homotopy class}\}.$$  
This map is called an \textit{extremal quasiconformal map} in the given homotopy class between two Riemann surfaces . For surfaces $F_g$ with genus $g>1$, the extremal quasiconformal map in certain special coordinates is locally an affine map except for some singularities, called the \textit{Teichm\"uller
 map}. The fundamental theorem about a Teichm\"uller
 map is the Teichm\"uller's theorem (see e.g.\cite{farb2011primer}).
\begin{theorem}
There exists a unique Teichm\"uller map in every homotopy class of homeomorphisms of $F_g$ with $g>1$ between two conformal structures on $F_g$.
\end{theorem}

\section{The Space of Shapes} 

We need to define rigorously the space of ``shapes" before constructing metrics on it. Various notions of shape spaces of curves and surfaces in $\mathbb{R}^2$ or $\mathbb{R}^3$ have been formulated from different perspectives with applications in computational geometry and computer graphics. An overview of various notions about shapes is given by Bauer, Bruveris and Minchor\cite{bauer2014overview}.

In this paper, we will introduce the space of shapes on surfaces from an intrinsic point of view. The idea uses the work by Ebin\cite{ebin1967space}, Fischer and Tromba\cite{fischer1984purely, tromba2012teichmuller}, Earle and Eells\cite{earle1969fibre}. We will summarize their work, define the shape space of a surface and complete the picture of its connection with the Teichm\"uller space. 

\subsection{Space of Riemannian metrics and its quotients} 
From the intrinsic viewpoint, the natural space to consider is the space of all smooth metric tensors on a given surface $F$, denoted by $\mathcal{M}$. Let $TF$ and $T^{*}F$ be the tangent and cotangent bundle, then a metric tensor is a section of $S^2T^{*}F$, the bundle of all symmetric (0,2)-type tensors. Since a metric tenor is positive definite, all metric tensors on $F$ form a convex subset of the infinite-dimensional vector space of sections of symmetric 2-tensors, denoted by $\Gamma(S^2T^{*}F)$. 

The tangent space at any element of $\mathcal{M}$, being a subset of a vector space, is naturally isomorphic to $\Gamma(S^2T^{*}F)$. In the tangent space at $g$ in $\mathcal{M}$, there is a natural inner product induced by $g$ on arbitrary tensor fields, defined as
$$(h, k)_{g} = \int_F tr_g(hk)dvol_g$$ 
where $h$ and $k$ are in $\Gamma(S^2T^{*}F)$ identified with the tangent space at $g$ and $dvol_g$ is the volume form. In local coordinates, they are represented by
$$tr_g(hk) = g^{ij}g^{lm}h_{il}k_{jm} \quad \text{   and   } \quad dvol_g = \sqrt{det (g)}dx_1dx_2.$$ 

Clarke\cite{clarke2010metric} explored the basic properties of this metric, showing that this metric, originally defined as a weak Riemannian metric, was indeed a metric. Furthermore, it coincides with the Weil-Petersson metric when restricted to the Teichm\"uller
 space.

It is hard to compute the natural $L^2$ metric defined above on the space $\mathcal{M}$. Besides, $\mathcal{M}$ contains redundant information: two metric tensors $h$ and $k$ may describe the isometric surface with different parametrizations. Therefore we would like to simplify the definition of the space of shapes for a given surface $F$, as the quotient of $\mathcal{M}$ by certain groups acting on $\mathcal{M}$.

There are three topological groups acting naturally on the space of metrics: the space $\mathcal{P}$ of all smooth functions on surface $F$, $\mathcal{D}$ the orientation-preserving diffeomorphism group of $F$ and its normal subgroup $\mathcal{D}_0$, the group of diffeomorphisms isotopic to the identity. The group $\mathcal{D}$  acts on $\mathcal{M}$ as isometries by pull-back
$$ \mathcal{D}\times\mathcal{M}\to\mathcal{M} \quad (f, g)\to f^*g.$$
The action of $\mathcal{D}_0$ is its restriction. The action of $\mathcal{P}$ on $\mathcal{M}$ is the multiplication of positive functions with metric tensors
$$ \mathcal{P}\times\mathcal{M}\to\mathcal{M} \quad (u,g)\to e^ug.$$

When we consider the two group actions above, an immediate question is whether we have a bundle structure. The natural topology for $\mathcal{M}$, $\mathcal{D}$ and $\mathcal{P}$ is the smooth Frechet topology, which means that two metrics are close if all the coefficients and their derivatives are close under the supremum norm in every chart. The implicit function theorem and its consequences are not true in general for this topology. Hence in Ebin\cite{ebin1967space} and Fischer\cite{fischer1977manifold, fischer1984purely}, $\mathcal{M}$, $\mathcal{D}$, and $\mathcal{P}$ are modelled in the corresponding Sobolev spaces.  These spaces contain maps which have square integrable partial derivatives up to sufficiently large order $s>1$ in every local charts, denoted by $\mathcal{M}^s$, $\mathcal{D}^{s+1}$, and $\mathcal{P}^s$ respectively. 

The space $\mathcal{M}^s$ forms an open convex subset in the Hilbert space $\Gamma^s(S^2T^{*}F)$, hence a Hilbert manifold. The space $\mathcal{P}^s$ corresponds to the Sobolev space $H^s(F,\mathbf{R})$. Then the multiplication and inverse are continuous, hence $\mathcal{P}^s$ is an abelian Hilbert Lie group. Ebin\cite{ebin1967space} proved that $\mathcal{D}^{s+1}$ was also a Hilbert Lie group. Then we can apply the following theorem in \cite{fischer1984purely} for the action of a Hilbert Lie group on an infinite-dimensional manifold, which will induce a smooth structure on the shape space, the space of pointwise conformal classes and the Teichm\"uller
 space.

\begin{theorem}
	Let a smooth Hilbert Lie group $\mathcal{G}$ act on a smooth Hilbert manifold $\mathcal{N}$. If the action is smooth, proper, and free, then:
	\begin{itemize}
		\item For all $x \in \mathcal{N}$, the orbit of $x$ by $\mathcal{G}$, denoted by $\mathcal{G}_x$, is a closed smooth submanifold in $\mathcal{N}$;
		\item The quotient space $\mathcal{N}/\mathcal{G}$ is a smooth manifold;
		\item The quotient map $\pi: \mathcal{N}\to\mathcal{N}/\mathcal{G}$ is a smooth submersion. It has the structure of a smooth principle fibre bundle.
	\end{itemize}
\end{theorem}

Fischer and Tromba\cite{fischer1984purely} considered the action of $\mathcal{P}^s$ on $\mathcal{M}^s$, where two metrics were in the same orbit if they differed by a factor $u\in \mathcal{P}^s$, namely they were pointwise conformal to each other. The quotient manifold of this group action on $\mathcal{M}^s$ is the space of pointwise conformal structures on $F$, denoted by $\mathcal{C}^s$. By Theorem 3.1 above, they clarified the differential structure for $\mathcal{C}^s$ in \cite{fischer1984purely}.

\begin{theorem}
	The group action $P: \mathcal{P}^s\times\mathcal{M}^s\to\mathcal{M}^s$ is smooth, free, and proper. The quotient space $\mathcal{C}^s = \mathcal{M}^s/\mathcal{P}^s$ by the quotient map $\pi: \mathcal{M}^s \to  \mathcal{M}^s/\mathcal{P}^s = \mathcal{C}^s$, is a contractible smooth Hilbert manifold, and $(\mathcal{M}^s, \mathcal{C}^s, \pi)$ has the structure of a trivial principle fiber bundle with structure group $\mathcal{P}^s$. The orbit $\mathcal{P}^s g$ for any $g$ is a closed smooth submanifold diffeomorphic to  $\mathcal{P}^s$. 
\end{theorem}	

For surfaces with genus at least two, there is a unique hyperbolic metric in each conformal class of metrics. Let $\mathcal{M}_{-1}$ and $\mathcal{M}_{-1}^s$ be the space of all smooth hyperbolic metrics and the corresponding Hilbert manifold, then Fischer and Tromba \cite{fischer1984purely} proved that $\mathcal{M}_{-1}^s$ and $\mathcal{C}^s$ were diffeomorphic, so we can use them interchangeably.

	We can take further quotient of $\mathcal{C}^s$ by group action of $\mathcal{D}_0^{s+1}$. This quotient gives a trivial fibre bundle description of the Teichm\"uller
 space $\mathcal{T}^s$ in \cite{earle1969fibre, fischer1984purely}.
	\begin{theorem}
	Assume a surface $F$ is of genus $g>1$. The group action $ \mathcal{D}_0^{s+1}\times\mathcal{C}^s \to \mathcal{C}^s$ by pullback is smooth, free, and proper. The quotient space is the Teichm\"uller
 space $\mathcal{T}^s$, and the quotient map $\pi: \mathcal{C}^s \to \mathcal{C}^s/\mathcal{D}_0^{s+1} = \mathcal{T}^s$ gives a trivial principle fibre bundle structure to $(\mathcal{C}^s, \mathcal{T}^s, \pi)$ . 
	\end{theorem}

The two groups can be combined to form a semidirect product $\mathcal{D}_0^{s+1} \ltimes \mathcal{P}^s$, which is called the \textit{conformorphism group} in Fischer\cite{fischer1977manifold} denoted by $\mathcal{E}_0^s$. It acts on $\mathcal{M}^s$ by 
$$\mathcal{E}_0^s\times\mathcal{M}^s\to\mathcal{M}^s \quad ((f, u), g)\to e^u\cdot f^*g;$$
$$(f_1, u_1)\cdot(f_2, u_2) = (f_2\circ f_1, e^{u_2+(u_1\circ f_2)}).$$

The quotient of the group action on $\mathcal{M}^s$ gives the Teichm\"uller
 space $\mathcal{T}^s$. This follows since $\mathcal{P}^s$ is a normal subgroup of $\mathcal{E}^s_0$ hence the two-step quotient $(\mathcal{M}^s/\mathcal{P}^s)/\mathcal{D}^s_0$ is isomorphic structure to $\mathcal{M}^s/\mathcal{E}^s_0$ \cite{fischer1977manifold}. In summary, we have the following diagram with two trivial fibre bundle structures

	\begin{displaymath}
	\begin{tikzcd}
\mathcal{M}^s \arrow[d,swap,"\mathcal{P}^s"] \arrow[dr,"\mathcal{E}_0^s"]& \quad \\
\mathcal{C}^s \arrow[r,"\mathcal{D}_0^{s+1}"] &   \mathcal{T}^s  
	\end{tikzcd}.
\end{displaymath}

Given these two trivial bundle structures, we can formally write $\mathcal{M}^s = \mathcal{P}^s\times\mathcal{D}^{s+1}_0\times \mathcal{T}^s$. It means that for any given metric $g \in \mathcal{M}^s$, there exist elements in $u\in\mathcal{P}^s$, $f\in\mathcal{D}^s$ and $[\tau]\in\mathcal{T}^s$ such that $g = e^u f^*(\sigma([\tau]))\in \mathcal{M}^s$. 
Here we don't have a canonical choice for a section $\sigma: \mathcal{T}^s \to \mathcal{M}_{-1}^s$, although a global section exists since the bundle is trivial.

\subsection{The space of shapes and its quotient}
Motivated by the definition of the Teichm\"uller space, we define the space of shapes as follows.
\begin{definition}
	Let $F$ be a closed orientable connected surface. The \textit{space of shapes} of $F$, or the \textit{shape space}, denoted by $\mathcal{S}(F)$, is the space of equivalence classes of metrics on the surface $F$, where two metrics $g_1$ and $g_2$ are equivalent if there exists an isometry $f: (F, g_1)\to (F, g_2)$ isotopic to the identity.  
\end{definition}
 This space is the quotient of $\mathcal{M}$ by the action of $\mathcal{D}_0$ as pullbacks. Alternatively, we can regard the elements in the shape space as equivalence classes of marked surfaces, denoted by $(F_i, \phi_i, g_i)$, where $F_i$ is a surface with metric $g_i$ diffeomorphic to $F$ via a marking $\phi_i: F_i\to F$. Two marked Riemannian surfaces $(F_1, \phi_1, g_1)$ and $(F_2, \phi_2, g_2)$ are equivalent if there exists an isometry $f: (F_1, g_1) \to (F_2, g_2)$ so that $ f \circ \phi_1$ is isotopic to $\phi_2$.

	We show that $\mathcal{S}^s$, the Hilbert manifold arising as the quotient manifold of the action by $\mathcal{D}_0^{s+1}$ on $\mathcal{M}^s$, has a principal bundle structure, which defines the differential structure on the shape space $\mathcal{S}^s$.

\begin{theorem}
	 The action by $\mathcal{D}^{s+1}_0$ on the space $\mathcal{M}^s$ is smooth, free, and proper if the surface $F$ has genus $g>1$. Hence the quotient space $\mathcal{S}^s = \mathcal{M}^s/\mathcal{D}^{s+1}_0$ is a smooth Hilbert manifold, and the quotient map $\pi: \mathcal{M}^s \to  \mathcal{M}^s/\mathcal{D}^{s+1}_0 = \mathcal{S}^s$ is smooth. $(\mathcal{M}^s, \mathcal{S}^s, \pi)$ has the structure of a principle fibre bundle with structure group $\mathcal{D}^{s+1}_0$. 
	 
\end{theorem}
\begin{proof}
		 The smoothness of the action of $\mathcal{D}^{s+1}$ on $\mathcal{M}^s$ was proved in detail by Ebin\cite{ebin1967space}. The properness of the action of $\mathcal{D}^{s+1}$ was given by Palais and Fischer (see, e.g. \cite{tromba2012teichmuller}) using a straightforward computation, so the same holds for the action of its normal subgroup $\mathcal{D}_0^{s+1}$. Hence we only need to prove that the action is free. We need to show that if $f^*g = g$ and $\mathcal{D}_0^{s+1}$, then $f$ has to be the identity.
	 
	 We prove it with harmonic maps. By Coron and Helein\cite{coron1989harmonic}, any smooth harmonic diffeomorphism between two compact Riemannian surfaces is a minimizer of the Dirichlet energy in its homotopy class, and it is unique if the genus is larger than 1. Hence for any metric $g_0$ on a surface $F$, if we have an isometry $f:(F, g_0) \to (F, g_0)$ isotopic to the identity, it has to be the identity by the uniqueness of harmonic maps, since the identity is a harmonic map.
\end{proof}

In the previous discussions, we consider all the spaces to be in the category of Hilbert manifolds 
$\mathcal{M}^s$, $\mathcal{S}^s$, $\mathcal{C}^s$, and $\mathcal{T}^s$ for sufficient large $s>0$, to guarantee the continuity of metric tensors and their derivatives. 
By choosing the category of the \textit{Inverse Limit Hilbert} structure, or ILH-structure, defined by Omori\cite{omori1970group}, the results above also hold for ILH-Lie groups $\mathcal{P}$, $\mathcal{D}_0$, and spaces $\mathcal{M}$, $\mathcal{M}_{-1}$, $\mathcal{S}$, $\mathcal{T}$(see, e.g.\cite{fischer1984purely}), so we will use this category in the rest of this paper. Notice that if the genus $g$ of $F$ is larger than one, the corresponding spaces $\mathcal{M}$ and $\mathcal{D}_0$ are contractible in this category, so the shape space $\mathcal{S}(F)$ is a contractible space, and the bundle structure $(\mathcal{M}, \mathcal{S}, \pi)$ is trivial.

Our next goal is to understand the connection between the space of shapes $\mathcal{S}$ and the Teichm\"uller space $\mathcal{T}$. There is a natural projection from $\mathcal{S}$ to  $\mathcal{T}$. By the Uniformization Theorem, there exists a unique hyperbolic metric $\bar{g}$ in the conformal class of $g$. The identity map $id: (F, g) \to (F, \bar{g})$ is conformal,  so we can define the following projection
$$	j: \mathcal{S} \to \mathcal{T} \quad [g]  \to [\bar{g}]. $$
\begin{lemma}
	The projection map $j: \mathcal{S} \to  \mathcal{T} $ is well-defined and smooth. 
\end{lemma}
\begin{proof}
Given $g_1$ and $g_2$ representing one equivalent class in $\mathcal{S}$ and their conformally equivalent hyperbolic metrics $\bar{g_1}$ and $\bar{g_2}$, we have an isometry $f: (F, g_1)\to (F, g_2)$ isotopic to the identity. It induces a conformal map from  $(F, \bar{g_1})$ to $(F, \bar{g_2})$ by $\bar{f} = id\circ f \circ id^{-1}$, since $id^{-1}$, $f$ and $id$ are conformal. Then $\bar{f}$ has to be an isometry since conformal diffeomorphisms between hyperbolic surfaces are isometries. Hence $\bar{g_1}$ and $\bar{g_2}$ represent the same element in the Teichm\"uller space.

This projection can be constructed explicitly using the bundle structure of $\mathcal{S}$. Since the bundle structure of $\mathcal{M}$ over $\mathcal{S}$ is trivial, there exists a smooth global section $\sigma: \mathcal{S} \to \mathcal{M}$. We can compose this section with the two smooth projections from $\mathcal{M} \to \mathcal{M}_{-1}$ and $\mathcal{M}_{-1} \to \mathcal{T}$ to construct the projection $j$. 
\end{proof}
Unfortunately we can't take the quotient of $\mathcal{S}$ by the group action of $\mathcal{P}$ directly to construct a well-defined group action. This is due to the fact that a function $u\in\mathcal{P}$ has a fixed value at a fixed point while every element in $\mathcal{S}$ can be represented using different metrics, which achieve possibly different values at a fixed point. It can also be seen by the fact that $\mathcal{D}_0$ is not a normal subgroup of $\mathcal{E}_0$, hence $\mathcal{E}_0/(\mathcal{D}_0, 1)$ is not isomorphic to $\mathcal{P}$ as groups.

In summary, we have four spaces $\mathcal{M}$, $\mathcal{S}$, $\mathcal{C}$, and $\mathcal{T}$ in a commutative diagram
\begin{displaymath}
	\begin{tikzcd}
\mathcal{M} \arrow[r,"\mathcal{D}_0"] \arrow[d,swap,"\mathcal{P}"]  \arrow[rd,"\mathcal{E}_0"] &
  \mathcal{S} \arrow[d,"j"] \\
\mathcal{C} \arrow[r,"\mathcal{D}_0"] & \mathcal{T}
	\end{tikzcd}. 
\end{displaymath}

The group action $\mathcal{D}\times\mathcal{M}\to\mathcal{M}$ is more subtle since certain metric tensors have non-trivial symmetries. For example, hyperbolic surfaces with genus $g$ may have isometry groups with order up to $84(g-1)$(see, e.g.\cite{farb2011primer}). 

The diagram above holds for surfaces $F$ with $g>1$. For the torus, its diffeomorphism group $\mathcal{D}_0$ could contain non-trivial isometries, so the action of $\mathcal{D}_0$ on $\mathcal{M}$ may not be free. By Earle and Eells\cite{earle1969fibre}, $\mathcal{D}_0$ is not contractible and has the same homotopy type as the torus, so the shape space $\mathcal{S}$ is not contractible. It does not fit in the picture for higher genus cases. Nevertheless, we define a metric structure on the shape space of a surface $F$, including the torus in the next section.

\section{Metrics on the Space of Shapes on Surfaces}
	In this section, we define a distance function between two shapes in the shape space $\mathcal{S}$ of a closed orientable surface $F$ of genus  $g\geq1$. We first discuss how to compare shapes using diffeomorphisms, then define a metric based on two energies defined for quasiconformal homeomorphisms on $F$. 
	
\subsection{Measurement of distortion}
To compare two shapes, we find an ``optimal" diffeomorphism between two shapes on a surface and measure its deviation from an isometry. In general, we can measure the distortion of $f: (F, g_1)\to (F, g_2)$ by the singular values of its differential, where the differential at a point $p$ is
$$ df_p : (T_pF, g_1) \to (T_{f(p)}F, g_2).$$
With an appropriate orthonormal basis in each metric, it can be expressed as 
\[
df_p = T = 
\begin{bmatrix}
    \lambda_1(p)     & 0 \\
    0     & \lambda_2(p)
\end{bmatrix}.
\]
where $\lambda_1(p)$ and $\lambda_2(p)$ are the singular values of $df_p$ as a linear transformation.
The area distortion of $f$ at $p$ is measured by the Jacobian $J_f(p) = \lambda_1(p)\lambda_2(p)$. The ratio of the two singular values at $p\in F$ corresponds to the eccentricity of the ellipse in the tangent space at $f(p)$ shown in Figure 1. To measure the angle distortion of $f$,  we define the \textit{dilatation} of $f$ at $p$ to be $K_f(p) = \lambda_1(p)/\lambda_2(p)$, assuming $\lambda_1(p) \geq \lambda_2(p)$.

Notice that we can extend these definitions from diffeomorphisms on $F$ to quasiconformal homeomorphisms on $F$. For a quasiconformal homeomorphism $f$ from a region $\Omega\subset\mathbb{C}$ into $\mathbb{C}$, $f_z$ and $f_{\bar{z}}$ are locally square-integrable, and $f$ is differentiable almost everywhere. The Jacobian $J_f$ is well-defined almost everywhere and locally integrable, and the essential supremum of $K_f$ over the surface is bounded. Then we can show that both $\lambda_1$ and $\lambda_2$ are locally square-integrable, satisfying the relations
$$\lambda_1(p) = \sqrt{J_f(p)K_f(p)} \quad \text{  and  } \quad \lambda_2(p) =\sqrt{\frac{J_f(p)}{K_f(p)}} \quad \forall p\in \Omega.$$

Since the Jacobian and dilatation of $f$ are local quantities, we can construct charts on a surface to show that $\lambda_1$ and $\lambda_2$ are well-defined and locally square-integrable for quasiconformal homeomorphisms on the surface $F$.

Based on the two singular values $\lambda_1$ and $\lambda_2$, we can define energies of $f$ measuring the angle distortion and the area distortion of $f$ respectively. 

\begin{definition}
	The \textit{area distortion energy} of a quasiconformal homeomorphism $f:(F,g_1)\to (F,g_2)$ is 
	$$E_{1}(f) = \sqrt{\int_{F}(1 - \sqrt{\lambda_1(p)\lambda_2(p)})^2dA_{g_1}}.$$
	The \textit{angle distortion energy} of $f$ is 
	$$E_{2}(f) = \frac{1}{2} ||\log \frac{\lambda_1(p)}{\lambda_2(p)}||_\infty.$$
	where $\lambda_1(p)$ and $\lambda_2(p)$ are singular values of $f$ at $p\in F$, and $||\cdot||_{\infty}$ is the essential supremum norm on the functions on $F$.
\end{definition}	
Note that if $f$ is a pointwise area-preserving, then $E_1(f) = 0$. If $f$ is conformal, $E_2(f) = 0$. Both of them are zero if and only if $f$ is an isometry. 

\subsection{Metric Structure for Genus Zero Surfaces}

Hass and Koehl\cite{Hass2017} introduced a metric structure for smooth genus-zero surfaces from the intrinsic point of view. By the Uniformization Theorem, any two metrics $g_1$ and $g_2$ on $S^2$ are conformally equivalent. There exists  a conformal diffeomorphism $f: (S^2, g_1) \to (S^2, g_2)$ with a positive function $\lambda_f$, called the \textit{conformal factor}, such that
	$$f^{*}(g_2) = \lambda_{f}^2 g_1 \quad \text{or} \quad g_2(f^{*}(v_1),f^{*}(v_2))_{f(p)}  =  \lambda^2_{f} (p)g_1(v_1,v_2)_p$$ 
	where $ v_1, v_2 \in T_pS^2$ for all $p\in S^2$. In looking for an energy minimizing map, we can restrict to the group of conformal diffeormophisms of the round 2-sphere, which coincides with the group of Mobius transformations isomorphic to $PSL(2,\mathbb{C})$. If we choose an appropriate orthonormal basis in the tangent space for each metric, the differential of a conformal diffeomorphism has a simple expression \cite{Hass2017}
	$$df_p = 
	\begin{bmatrix}
    \lambda_f       & 0  \\
    0     & \lambda_f   \\
    \end{bmatrix}.$$  
	For conformal maps we have $E_2 = 0$ and $E_1$ simplifies to 
	$$E_1(f) = \sqrt{\int_{S^2}(1 - \lambda_{f})^2 dA_{g_1}}.$$

This idea leads to the definition of a metric on the space of shapes of $S^2$ as
	$$d((S^2, g_1),(S^2, g_2)) = \text{inf}\{E_1(f) | f:(S^2, g_1)\to(S^2,g_2)  \text{ a conformal diffeomorphism}\}.$$
	
	In \cite{Hass2017}, Hass and Koehl showed this function $d: \mathcal{S}\times \mathcal{S}\to \mathbb{R}$ gave a metric, and the infimum was achieved by a conformal diffeomorphism. In their framework, the given two surfaces are mapped to the round 2-sphere by conformal maps $c_1$ and $c_2$. They found an optimal conformal diffeomorphism $c_2^{-1}\circ m\circ c_1$ between the two surfaces by minimizing the symmetric distortion energy among the group of Mobius transformations. They proposed an  algorithm to compute the distance between two triangulated surfaces and applied it to describe shapes of proteins and generate evolutionary trees of species \cite{hass2014round, koehl2013automatic, koehl2015landmark}. 
	
\begin{figure}[h!]
  \includegraphics[width=0.6\linewidth]{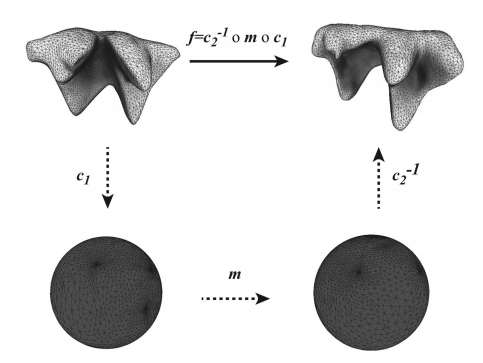}
  \caption{A framework to compare genus-zero surfaces. Picture is a courtesy of Hass and Koehl \cite{Hass2017}.}
\end{figure}

\subsection{Metric Structure for High Genus Surfaces}
There is a fundamental difference between the shape space of genus-zero surfaces $S^2$ and that of higher genus surfaces $F$. Any two shapes on the 2-sphere are conformal, while two shapes on a high genus surface are not necessarily conformally equivalent. We define a distance between two shapes by minimizing the sum of the energies $E_1$ and $E_2$ over the quasiconformal homeomorphisms of $F$ isotopic to the identity. Setting $E(f) = E_1(f) + E_2(f)$, we define a distance function as follows.
\begin{definition} Let $F$ be a closed connected orientable surface of genus $g\geq 1$ and $\mathcal{S}(F)$ be the shape space of $F$. Then we define a function $d: \mathcal{S}(F)\times\mathcal{S}(F) \to \mathbb{R}$ between two shapes in $\mathcal{S}(F)$ represented by $(F, g_1)$ and $(F, g_2)$ to be

$$d((F, g_1), (F, g_2)) = \inf_{f\in \mathcal{Q}_0} {E(f)} = \inf_{f\in \mathcal{Q}_0}\bigg(\sqrt{\int_{F}(1 - \sqrt{\lambda_1\lambda_2})^2dA_{g_1}} +  \frac{1}{2}||\log \frac{\lambda_1}{\lambda_2}||_\infty)$$
where $\mathcal{Q}_0$ is the space of quasiconformal homeomorphisms from $(F, g_1)$ to $(F, g_2)$ isotopic to the identity. 
\end{definition}
Equivalently, we can use marked surfaces to define this metric on the shape space $\mathcal{S}(F)$. Let $(F_1, \phi_1, g_1)$ and $(F_2, \phi_2, g_2)$ represent two different shapes of $F$, then 
$$d((F_1, \phi_1, g_1), (F_2, \phi_2, g_2)) = \inf_{f\in \mathcal{Q}} \{E(f)\} $$
where $\mathcal{Q}$ is the set of quasiconformal homeomorphisms from $(F_1, g_1)$ to $(F_2, g_2)$ isotopic to $\phi_2\circ\phi_1^{-1}$.

In the rest part of this section, we will show that the function $d$ is a distance function on the shape space $\mathcal{S}(F)$, and the energy minimizer is realized by a quasiconformal homeomorphism between two surfaces. We will first prove the existence of the minimizer based on the lower semi-continuity of the energy. 

In general, a sequence of homeomorphisms $f_n$ of a surface may converge to a singular map, such as a constant map. We show that singular maps will not occur for the limit of an energy-minimizing sequence. 

Given two hyperbolic surfaces $(F, \bar{g}_1)$ and $(F, \bar{g}_2)$, all $K$-quasiconformal homeomorphisms between them are equicontinuous. (See Theorem 4.4.1 in \cite{hubbard2006teichmuller}.) The following lemma shows that this result also holds for $K$-quasiconformal homeomorphisms between two flat tori. To prove this lemma, we use the extremal length of curve families in the annulus (see e.g. \cite{fletcher2007quasiconformal}).

\begin{lemma}
Let $f_n:(\mathbb{T}^2, g_1) \to (\mathbb{T}^2, g_2)$ be a family of $K$-quasiconformal homeomorphisms between two flat tori with unit area. Then the maps $f_n$ are equicontinuous. 
\end{lemma}
\begin{proof}
Let $J$ be the injective radius of $(\mathbb{T}^2, g_1)$, and $d_{g_i}(x, y)$ denote the distance between $x$ and $y$ in the metric $g_i$, where $i = 1, 2$. Then for any $0<r<J$, if $d_{g_1}(x, y)< r$, then there exists an embedded annulus $A$ in $(\mathbb{T}^2, g_1)$ centered at the midpoint of $x$ and $y$, whose inner radius is $r/2$ and  outer radius is $J/2$. Moreover,  it separates $\mathbb{T}^2$ into two components, one of which is a flat disk with radius $r/2$ containing $x$ and $y$. 

Lift $A$ isometrically to a flat annulus $\tilde{A}$ in the universal covering $\mathbb{R}^2$, and lift $x$ and $y$ to $\tilde{x}$ and $\tilde{y}$ contained in the disk bounded by the inner boundary of $\tilde{A}$. We consider the extremal length $\lambda(\Gamma)$ of the family of curves $\Gamma$ in $\tilde{A}$ that separate the two boundary circles of $\tilde{A}$, with curves not leaving $\tilde{A}$. Then we have (see e.g.\cite{fletcher2007quasiconformal})
$$\lambda(\Gamma)  = \frac{2\pi}{\log(J/r)}.$$
We also lift $f_n$ to $K$-quasiconformal homeomorhisms $\tilde{f}_n: \mathbb{R}^2 \to \mathbb{R}^2$.
Then by the property of $K$-quasiconformal homeomorphisms, $\tilde{f}_n(A)$ are annulus, and if $\Gamma_1^n = \tilde{f}_n(\Gamma)$, then the curves in $\Gamma_1^n$ are contained in $\tilde{f}_n(A)$ with their extremal length bounded by 
$$\lambda(\Gamma_1^n) \leq K\lambda(\Gamma).$$
By the definition of the extremal length $\lambda(\Gamma)$,  notice that the area of $\tilde{f}_n(A)$ is less than one in the Euclidean metric on $\mathbb{R}^2$, so
$$\lambda(\Gamma_1^n) \geq  L^2 \geq 4d^2(\tilde{f}_n(\tilde{x}), \tilde{f}_n(\tilde{y}))=4d^2_{g_2}(f_n(x), f_n(y)) \quad \forall n,$$ 
where $L$ is the length of the inner boundary curve of $\tilde{f}_n(A)$. The second inequality holds because $\tilde{f}_n$ is a homeomorphism so that $\tilde{f}_n(\tilde{x})$ and $\tilde{f}_n(\tilde{y})$ are in the disk bounded by the inner boundary curve of $\tilde{f}_n(A)$, and the last equality holds because there exists an isometric project from $\mathbb{R}^2$ to $(\mathbb{T}^2, g_2)$. Then we conclude 
$$d_{g_2}(f_n(x), f_n(y)) \leq \sqrt{\frac{\pi K}{2\log\frac{J}{r}}}\quad \forall n.$$
Notice that $d_{g_2}(f(x), f(y)) \to 0$ if $r\to 0$. Hence for any $\epsilon>0$, there exists $r>0$ such that if $d_{g_1}(x,y) <r$, then $d_{g_2}(f(x), f(y)) < \epsilon$. Notice that $r$ doesn't depend on $n$, hence the maps $f_n$ are equicontinuous.

\end{proof}

\begin{theorem}

Assume $F$ has genus $g\geq 1$. Given two metrics $(F, g_1)$, $(F, g_2)$ representing two shapes in $\mathcal{S}(F)$, and an energy-minimizing sequence $f_n \in \mathcal{Q}_0(F)$ such that $E(f_n) \to d((F, g_1), (F, g_2))$ as $n\to \infty$, there is a subsequence of $f_n$ converging to a quasiconformal homeomorphism $f$ such that $E(f) = d((F, g_1), (F, g_2))$. 
\end{theorem}

\begin{proof}
Since $E(f_n) \to d((F, g_1), (F, g_2))$, we assume that $E(f_n) < K$ for some $K>0$. Then the maps $f_n$ are $K$-quasiconformal homeomorphisms on a compact surface $F$.   Then the maps $f_n$ are equicontinuous and bounded with respect to the corresponding metrics $\bar{g}_1$ and $\bar{g}_2$ of constant curvature, and if $F$ is the torus, we normalize $\bar{g}_1$ and $\bar{g}_2$ to be metrics with unit area. By Arzela-Ascoli, there exists a subsequence converging uniformly to a continuous map $f$. To show $f$ is a homeomorphism, notice that the inverses of the maps $f_n$ are also $D$-quasi-isometries where $D$ does not depend on $n$. Then the equicontinuity of inverses of $f_n$ implies that if $f^{-1}_n(x) = a$ and $f^{-1}_n(y) = b$, then 
$$ d_{\bar{g}_1}(a, b) = d_{\bar{g}_1}(f_n^{-1}(x), f_n^{-1}(y)) \leq C(K)d_{\bar{g}_2}(x, y) = C(K)d_{\bar{g}_2}(f_n(a), f_n(b))$$
where $d_{\bar{g}_i}(x, y)$ denotes the distance between $x$ and $y$ in the metric $\bar{g}_i$ for $i = 1, 2$. Taking the limit $n\to \infty$, we conclude that $f$ is injective. Then $f$ is a continuous injection from a compact 2-manifold to a connected 2-manifold, so it is a homeomorphism by the properness of $f$ and the theorem of invariance of domain. (see e.g.\cite{lee2010introduction}).

Replace $f_n$ by a convergent subsequence and we have $f_n \to f$ uniformly where $f$ is a homeomorphism. For the limit map $f$, notice that its energy is given by
$$E(f)= \sqrt{\int_{F}(1 - \sqrt{J_f})^2dA_{g_1}} +  \frac{1}{2}\log K_f$$
where $J_f$ is the Jacobian of $f$ and $K_f$ is the maximal dilatation of $f$. The lower semicontinuity property of the maximal dilatations for quasiconformal maps \cite{bojarski2013infinitesimal} gives
$$K_f \leq \liminf_{n\to \infty} K_{f_n}.$$

Next we will show that by taking a further subsequence of $f_n$, we have 

$$\int_F (1 - \sqrt{J_f})^2dA_{g_1} = \lim_{n\to \infty}\int_F (1 - \sqrt{J_{f_n}})^2dA_{g_1}.$$

Notice that this term has the following decomposition;
\begin{align*}
\int_F (1 - \sqrt{J_{f_n}})^2dA_{g_1} & =  \int_F dA_{g_1} + \int_F J_{f_n}dA_{g_1} - 2\int_F\sqrt{J_{f_n}} \\
  & =  Area((F, g_1)) + Area((F, g_2)) - 2\int_F \sqrt{J_{f_n}}, 
\end{align*}
where $Area((F, g))$ is the area of the surface $F$ with respect to metric $g$. Similarly we have 

$$\int_F (1 - \sqrt{J_{f}})^2dA_{g_1}  =  Area((F, g_1)) + Area((F, g_2)) - 2\int_F \sqrt{J_{f}}. $$

Consider $\sqrt{J_{f_n}}$ as an element in the function space on $(F, g_1)$ with $L^2$ norm. The area of $(F, g_2)$ gives a uniform bound;

$$\int_F (\sqrt{J_{f_n}})^2 dA_{g_1} = Area((F, g_2)) = \int_F(\sqrt{J_f})^2 dA_{g_1}.$$

The unit closed ball in the function space on $(F, g_1)$ with $L^2$ norm is weakly sequentially compact, so we have a subsequence of $f_n$, denoted again by $f_n$ such that $\sqrt{J_{f_n}}$ converges weakly to $\sqrt{J_f}$. Since $(F, g_1)$ is compact, constant functions are in this function space, hence

$$\lim_{n\to\infty} \int_F \sqrt{J_{f_n}}\cdot 1 dA_{g_1} = \int_F \sqrt{J_{f}}\cdot 1 dA_{g_1}.$$

Thus, we have 
$$E(f) \leq \liminf_{n \to \infty} E(f_n) = d((F, g_1), (F, g_2)).$$
Since $f$ is a quasiconformal homeomorphsim, $E(f) \geq d((F, g_1), (F, g_2))$, hence  $$E(f) = d((F, g_1), (F, g_2)).$$
\end{proof}

We are ready to check that $d$ satisfies the conditions for a distance function. 

\begin{theorem}
	Let $F$ be a closed orientable connected surface of genus $g\geq 1$. The function $d$ induces a metric on the space of shapes $\mathcal{S}(F)$.
\end{theorem}
\begin{proof}
To show the function $d$ is a metric,  we need to check that for any three metrics $(F, g_1)$, $(F, g_2)$, and $(F, g_3)$, we have
\begin{itemize}
	\item[(1)]  $d((F, g_1),(F, g_2))\geq 0$; 
	\item[(2)]  $d((F, g_1),(F, g_2))=0$ if and only if $g_1$ and $g_2$ are isometric by a diffeomorphism isotopic to the identity;
	\item[(3)]  $d((F, g_1),(F,g_2)) = d((F, g_2),(F,g_1))$;
	\item[(4)]  $d((F, g_1),(F, g_3)) \leq d((F, g_1),(F, g_2)) + d((F, g_3),(F, g_2))$.
\end{itemize}

	The first property is immediate. If two surfaces are isometric, both singular values of the differential are one at every point on the surface, hence the distance is zero. If $d((F, g_1),(F, g_2))=0$, by Theorem 4.4, there exists a quasiconformal $f$ homeomorphism realizing this energy, then $f$ is a $1$-quasiconformal homeomorphism, hence is a conformal. Moreover, the area distortion is zero, so it is an isometry isotopic to the identity, which means that $g_1$ and $g_2$ represent the same equivalence class in $\mathcal{S}(F)$.
	
	The symmetry property follows from $E_1(f) = E_1(f^{-1})$ and $E_2(f) = E_2(f^{-1})$. By a similar computation in \cite{Hass2017}, we have 

\begin{align*} 
E_1(f^{-1}) &= \sqrt{\int_F (1 - \sqrt{\frac{1}{\lambda_1\lambda_2}})^2 dA_{g_2}} = \sqrt{\int_F (1- \sqrt{\frac{1}{\lambda_1\lambda_2}})^2 \lambda_1\lambda_2 dA_{g_1}} \\ 
\quad  &= \sqrt{ \int_F (1- \sqrt{\lambda_1\lambda_2})^2 dA_{g_1}} = E_1(f).
\end{align*}
The singular values of $f^{-1}$ are  $1/
\lambda_1$ and $1/\lambda_2$, so the symmetry of $E_2$ is immediate. 
	
	To show the triangle inequality, set $f:(F, g_1)\to(F, g_2)$ and $g: (F, g_2)\to(F, g_3)$, and we show that 
$$E_1(g\circ f) \leq E_1(g) + E_1(f).$$
Let the singular values of $f$, $g$, and $g\circ f$ be $\lambda_1$ and $\lambda_2$, $\mu_1$ and $\mu_2$, $\sigma_1$ and $\sigma_2$ respectively. Then by a similar computation in \cite{Hass2017}, we have
 \begin{align*}
 (E_1(g) + E_1(f))^2 &= \int_F (1-\sqrt{\lambda_1\lambda_2})^2 dA_{g_1} + \int_F (1-\sqrt{\mu_1\mu_2})^2 dA_{g_2} \\ 
 \quad &+ 2 \sqrt{\int_F(1-\sqrt{\lambda_1\lambda_2})^2dA_{g_1}\int_F (1-\sqrt{\mu_1\mu_2})^2 dA_{g_2}}.
 \end{align*}
Notice that $dA_{g_2} = \lambda_1\lambda_2dA_{g_1}$, then by the Cauchy-Schwarz inequality, we have
\begin{align*}
&\sqrt{\int_F(1-\sqrt{\lambda_1\lambda_2})^2dA_{g_1}\int_F (1-\sqrt{\mu_1\mu_2})^2 dA_{g_2}}\\ &=\sqrt{\int_F(1-\sqrt{\lambda_1\lambda_2})^2dA_{g_1}\int_F (1-\sqrt{\mu_1\mu_2})^2 \lambda_1\lambda_2 dA_{g_1}} \\&
 \geq \int_F(1-\sqrt{\lambda_1\lambda_2})(1-\sqrt{\mu_1\mu_2})\sqrt{\lambda_1\lambda_2}dA_{g_1}.
\end{align*}
Hence 
\begin{align*}
(E_1(g) + E_1(f))^2 &\geq \int_F(1-\sqrt{\lambda_1\lambda_2})^2 + (1-\sqrt{\mu_1\mu_2})^2\lambda_1\lambda_2 \\
  &+ 2(1-\sqrt{\lambda_1\lambda_2})(1-\sqrt{\mu_1\mu_2})\sqrt{\lambda_1\lambda_2}dA_{g_1}\\
 & = \int_F ((1-\sqrt{\lambda_1\lambda_2})+ \sqrt{\lambda_1\lambda_2}(1- \sqrt{\mu_1\mu_2}))^2dA_{g_1}\\
 & = \int_F (1 - \sqrt{\lambda_1\lambda_2\mu_1\mu_2})^2dA_{g_1}.
\end{align*}
Since $\sigma_1\sigma_2 = J_{g\circ f} = J_fJ_g = \lambda_1\lambda_2\mu_1\mu_2$, it follows that
$$(E_1(g) + E_1(f))^2 \geq (E_1(g\circ f))^2.$$
To prove the second part of the inequality, namely $E_2(g\circ f) \leq E_2(g) + E_2(f)$, we assume that $\lambda_1\geq \lambda_2$, $\mu_1 \geq \mu_2$, and $\sigma_1 \geq \sigma_2$ for simplicity. Notice that the larger singular value is the 2-norm for the differential $df_p$, and the smaller singular value is the reciprocal of the 2-norm of the inverse of $df_p$. The larger singular value of the composition $g\circ f$ is bounded by
$$\sigma_1(p)  = ||d(g\circ f)_p||_2 = ||dg_{f(p)}\circ df_{p}||_2 \leq ||df_p||_2 ||dg_{f(p)}||_2=\lambda_1(p) \mu_1(p). $$  
Similarly for the inverse, we have
$$ \frac{1}{\sigma_2(p)}  = ||d(f\circ g)^{-1}_p||_2 = ||df^{-1}_p\circ dg^{-1}_{f(p)}||_2 \leq ||df^{-1}_p||_2 ||dg^{-1}_{f(p)}||_2=\frac{1}{\lambda_2(p) \mu_2(p)}. $$  
Hence we have 
$$0<\lambda_2\mu_2\leq \sigma_2\leq \sigma_1\leq \lambda_1\mu_1.$$
Therefore
\begin{align*}
E_2(f)+ E_2(g) &= \frac{1}{2}||\log \frac{\lambda_1}{\lambda_2}||_\infty + \frac{1}{2}||\log \frac{\mu_1}{\mu_2}||_\infty\geq \frac{1}{2}||\log \frac{\lambda_1}{\lambda_2}+\log \frac{\mu_1}{\mu_2}||_\infty \\
&= \frac{1}{2}||\log \frac{\lambda_1\mu_1}{\lambda_2\mu_2}||_\infty \geq \frac{1}{2}||\log\frac{\sigma_1}{\sigma_2}||_\infty = E_2(g\circ f).
\end{align*}
Therefore we show that 
$$E(f) + E(g) \geq E(g\circ f).$$
To pass to the infimum, we  choose $f_n:(F,g_1)\to(F,g_2)$ and $g_n:(F,g_2)\to(F,g_3)$ in $\mathcal{Q}_0$ such that 
$$\lim_{n\to \infty}E(f_n) = d((F,g_1), (F, g_2)) \text{ and } \lim_{n\to \infty}E(g_n) = d((F,g_2), (F, g_3)).$$
Then we have 
$$E(f_n) + E(g_n)\geq E(g_n\circ f_n)\geq d((F,g_1),(F, g_3)).$$
Taking the limit as $n \to \infty$ we have 
$$d((F,g_1),(F, g_2)) + d((F,g_2),(F, g_3))\geq d((F,g_1),(F, g_3)).$$
The last thing to check is that the metric $d$ is well-defined on the shape space. Assume $g_1$ and $ \tilde{g}_1$ represent the same shape, and $g_2$ and $\tilde{g}_2$ represent another shape. Then we have an isometry $i_1: (F, g_1)\to (F, \tilde{g}_1)$ isotopic to the identity and another isometry $i_2: (F, g_2)\to (F, \tilde{g}_2)$ isotopic to the identity. Given $f:(F,g_1)\to(F,g_2)$, consider the map $\tilde{f}:(F, \tilde{g}_1)\to (F, \tilde{g}_2)$ defined as 
$$\tilde{f} = i_2\circ f \circ i_1^{-1}.$$
Since $i_1$ and $i_2$ are isometries, they will not change the singular values, so the singular values of $\tilde{f}$ are given by $\tilde{\lambda}_1(p) = \lambda_1(i_1^{-1}(p))$ and $\tilde{\lambda}_2(p) = \lambda_2(i_1^{-1}(p))$. An isometry also preserves the area, so $dA_{\tilde{g}_1}$ = $dA_{g_1}$. Hence we have 

\begin{align*}
E_1(\tilde{f}) &= \sqrt{\int_F (1-\sqrt{\tilde{\lambda}_1(p)\tilde{\lambda}_2(p)})^2dA_{\tilde{g}_1}} \\ &= \sqrt{\int_F (1-\sqrt{\lambda_1(i_1^{-1}(p))\lambda_2(i_1^{-1}(p))})^2dA_{g_1}} = E_1(f).
\end{align*}

\noindent and 

$$E_2(\tilde{f}) = \frac{1}{2}||\log \frac{\tilde{\lambda}_1}{\tilde{\lambda}_2}||_\infty = \frac{1}{2}||\log \frac{\lambda_1}{\lambda_2}||_\infty = E_2(f).$$
Hence we have 
$$E(\tilde{f}) = E(f). $$
Since $i_1$ and $i_2$ are isotopic to the identity, $f\in \mathcal{Q}_0$ if and only if $\tilde{f} \in \mathcal{Q}_0$. Taking the infimum over $f\in \mathcal{Q}_0$, we conclude that  
$$d((F,g_1),(F,g_2)) =  d((F, \tilde{g}_1),(F, \tilde{g}_2)).$$ 
Hence $d$ is a well-defined metric on $\mathcal{S}$.
\end{proof}
Notice that if we restrict the metric to $\mathcal{T}$, then $d$ will be the Teichm\"uller
 metric.

It is not clear whether the minimizer is unique between two general surfaces.  In the special case where both surfaces $(F, g_1)$ and $(F, g_2)$ are flat tori with unit area, the minimizers are given by affine maps, because affine maps coincide with Teichm\"uller maps on flat tori with unit area, and the Jacobians of affine maps are constant. This forces the Jacobians to be the constant $J \equiv 1$ on $F$. If we fix one point $p$ on $F$, then there is a unique affine map fixing $p$ realizing the infimum of the energy.

\section{Conclusion and Further Work}
We have described a new metric structure on the shape space of a high-genus surface. We first define the shape space of a surface and establish its connections with the Teichm\"uller space. Then we introduce an energy for quasiconformal maps as a measurement of distortion, and define a distance function on the shape space by minimizing this energy among all the quasiconformal homeomorphisms in a given homotopy class of maps between two given shapes. We prove that the minimizer of this energy is a quasiconformal homeomorphism, which produces an optimal correspondence between two shapes. 

In the future, we will design an algorithm to compute the distance between two shapes represented by triangulated surfaces.  The framework of the algorithm in \cite{li2008globally,lui2014geometric, wong2014computation} can be adapted to our case.   Also, the uniqueness of the energy-minimizing map is open. 

\section{Acknowledge}
The author would like to thank his advisors, Joel Hass and Patrice Koehl, for suggesting this problem and constant discussions and encouragement.

\bibliography{ref} 
\bibliographystyle{amsplain}

\end{document}